\documentclass[12pt]{article}
\usepackage[
        a4paper,
        left=2.5cm,
        right=2.6cm,
        top=3cm,
        bottom=4cm,
        	vmargin=2cm
]{geometry}

\usepackage{textcomp}
\usepackage{graphicx}

\usepackage{amssymb,amsmath,amsthm,color,mathtools,thmtools}
\usepackage{float}
\usepackage[font=small]{caption}
\usepackage[font=footnotesize]{subcaption}
\usepackage{tikz}
\usepackage{hyperref}
\hypersetup{colorlinks=true,citecolor=blue, linkcolor=blue,
urlcolor=blue}
\usepackage{comment}
\usepackage[shortlabels]{enumitem}
\usepackage{cleveref}
\usepackage{todonotes}

\usepackage[square,numbers]{natbib}
\AtBeginDocument{\bibsep=5pt}

\newcommand{\cols}{\calc}

\newtheorem{theorem}{Theorem}[section]
\newtheorem*{theorem*}{Theorem}
\newtheorem{lemma}[theorem]{Lemma}
\newtheorem*{lemma*}{Lemma}
\newtheorem{claim}[theorem]{Claim}
\newtheorem*{claim*}{Claim}
\newtheorem{proposition}[theorem]{Proposition}
\newtheorem*{proposition*}{Proposition}

\newtheorem{corollary}[theorem]{Corollary}
\newtheorem*{corollary*}{Corollary}

\newtheorem{definition}[theorem]{Definition}

\newcommand{\dprm}{^{\prime \prime}}

\newcommand{\abs}[1]{\left|#1\right|}

\newcommand{\nats}{\mathbb{N}}

\newcommand{\calc}{\mathcal{C}}

\newcommand{\tendsto}{\rightarrow}

\newcommand{\e}{{\rm e}}

\newcommand{\calH}{\mathcal{H}}

\newcommand{\Vflx}{V_{\text{flex}\,}}

\newcommand{\Vbuf}{V_{\text{buf}\,}}

\newcommand{\colsflx}{\cols_{\text{flex}}}

\newcommand{\colsbuf}{\cols_{\text{buf}\,}}

\newcommand{\gnp}[2]{\mathbf{G}(#1,#2)}

\newcommand{\eps}{\varepsilon}

\newcommand{\rg}{\mathbf{G}}
\newcommand{\rH}{\textbf{H}}
\newcommand{\rD}{\textbf{D}}

\def \cG {\mathcal{G}}

\newcommand{\cC}{\mathcal{C}}
\newcommand{\cD}{\mathcal{D}}
\newcommand{\cP}{\mathcal{P}}

\usepackage{tcolorbox}

\newtcolorbox{kyriakosbox}{
        colframe=cyan!20!white,
        colback =cyan!20!white,
        top=0mm, bottom=0mm, left=0mm, right=0mm,
        arc=0mm,
        fontupper=\color{blue!70!black},
        fonttitle=\bfseries\color{blue!70!black},
        title=Kyriakos:
                        }

\newcommand{\dnp}[2]{\mathbf{D}(#1,#2)}

\newcommand{\PP}{\mathbb{P}}
\newcommand{\EE}{\mathbb{E}}

\DeclareMathOperator{\flex}{flex}
\DeclareMathOperator{\buf}{buf}

\usepackage{tcolorbox}
\usepackage{xcolor}

\title{Rainbow subgraphs of uniformly coloured randomly perturbed graphs}
\author{
	Kyriakos Katsamaktsis \thanks{
	Department of Mathematics, 
	University College London, 
	Gower Street, London WC1E~6BT, UK. \\
	Email: \texttt{\{kyriakos.katsamaktsis.21|s.letzter\}}@\texttt{ucl.ac.uk}.
    The research of KK is supported by the Engineering and Physical Sciences Research Council [grant number EP/W523835/1], the research of
	SL is supported by the Royal Society. 
	}
    \and Shoham Letzter\footnotemark[1]
    \and
    Amedeo Sgueglia\thanks{
	Fakult\"at f\"ur Informatik und Mathematik,
    Universit\"at Passau, Germany. \\
	Email: \texttt{amedeo.sgueglia}@\texttt{uni-passau.de}.
    While conducting this research, AS was affiliated with University
College London and supported by the Royal Society.}
}

\date{}

\begin{document}

\maketitle
\begin{abstract}
\setlength{\parskip}{\medskipamount}
\setlength{\parindent}{0pt}
\noindent
    For a given $\delta \in (0,1)$, the randomly perturbed graph model is defined as the union of any $n$-vertex graph $G_0$ with minimum degree $\delta n$ and the binomial random graph $\mathbf{G}(n,p)$ on the same vertex set. 
    Moreover, we say that a graph is uniformly coloured with colours in $\mathcal{C}$ if each edge is coloured independently and uniformly at random with a colour from $\mathcal{C}$.

    Based on a coupling idea of McDiarmid, we provide a general tool to tackle problems concerning finding a rainbow copy of a graph $H=H(n)$ in a uniformly coloured perturbed $n$-vertex graph with colours in $[(1+o(1))e(H)]$.
    For example, our machinery easily allows to recover a result of Aigner-Horev and Hefetz concerning rainbow Hamilton cycles, and to improve a result of Aigner-Horev, Hefetz and Lahiri concerning rainbow bounded-degree spanning trees.

    Furthermore, using different methods, we prove that for any $\delta \in (0,1)$ and integer $d \ge 2$, there exists $C=C(\delta,d)>0$ such that the following holds.
    Let $T$ be a tree on $n$ vertices with maximum degree at most $d$ and $G_0$ be an $n$-vertex graph with $\delta(G_0)\ge \delta n$.
    Then a uniformly coloured $G_0 \cup \mathbf{G}(n,C/n)$ with colours in $[n-1]$ contains a rainbow copy of $T$ with high probability.
    This is optimal both in terms of colours and edge probability (up to a constant factor). 
\end{abstract}

\section{Introduction}
    Given \(\delta \in (0,1)\), we define \(\cG_{\delta,n}\) to be the family of graphs on vertex set $[n]$ with minimum degree at least $\delta n$, and we let \(\gnp{n}{p}\) be the binomial random graph on vertex set $[n]$ with edge probability \(p\).
    One of the central themes in extremal combinatorics is determining the minimum degree threshold for a given graph property $\cP$, i.e.\ how large \(\delta\) needs to be so that every \(G \in \cG_{\delta,n}\) satisfies $\cP$.
    Similarly, probabilistic combinatorics aims to determine how large $p$ needs to be for \(\gnp{n}{p}\) to satisfy $\cP$ with high probability\footnote{
    Formally, we say that a sequence of events \((A_n)_{n\in \nats}\) holds \emph{with high probability} if \(\PP[A_n] \tendsto 1\) as \(n\tendsto \infty\).
    }.
    Bohman, Frieze and Martin~\cite{bohman} provided a connection between the extremal and the random graph settings by introducing the \emph{randomly perturbed graph} model.
    For a given $\delta \in (0,1)$, this is defined as \(G_0 \cup \gnp{n}{p}\) where \(G_0 \in \cG_{\delta,n}\), i.e.\ as the graph on \([n]\) whose edge set is the union of the edges of a deterministic graph \(G_0\) with minimum degree at least $\delta n$ and the edges of a random graph \(\gnp{n}{p}\) on the same vertex set.
    For a given $\delta$ and a given graph property $\cP$, a pivotal question in the area is to determine how large $p$ needs to be so that for every $G_0 \in \cG_{\delta,n}$, with high probability, $G_0 \cup \gnp{n}{p}$ satisfies $\cP$.
    More precisely, we say that $\hat{p}=\hat{p}(\delta,\cP,n)$ is a \emph{perturbed threshold} for the property $\cP$ at $\delta$ if there are constants $C > c > 0$ such that for any $p \ge C \hat{p}$ and for any sequence of $n$-vertex graphs $(G_{n})_{n \in \mathbb{N}}$ with $G_{n} \in \cG_{\delta,n}$ we have $\lim_{n\to\infty}\PP\big(G_{n} \cup G(n,p) \in \cP \big)=1$, and for any $p \le c \hat{p}$ there exists a sequence of $n$-vertex graphs $(G_{n})_{n\in\mathbb{N}}$ with $G_n \in \cG_{\delta,n}$ such that $\lim_{n\to\infty}\PP\big( G_{n} \cup G(n,p)\in \cP\big) =0$.
    
	For example, the main result of~\cite{bohman} is that when $\cP$ is the property of being Hamiltonian, $n^{-1}$ is a perturbed threshold for $\cP$ at $\delta$, for any $\delta \in (0,1/2)$.
    This interpolates between the well known result
    that the threshold for the containment of a Hamilton cycle in $\gnp{n}{p}$ is $n^{-1}\log n$, and the classical theorem of Dirac that every $n$-vertex graph with minimum degree at least $n/2$ is Hamiltonian (and thus when $\delta \ge 1/2$, no random edges are needed in the perturbed model).
    Since~\cite{bohman}, there has been a sizeable body of research extending and adapting results from the extremal and the probabilistic setting to the perturbed one, particularly when the property $\cP$ is the containment of a spanning subgraph (e.g.\ trees~\cite{tree_universality,joos_kim_trees,kriv-kwan-sudakov}, factors~\cite{balogh2019tilings,BPSS_triangles,han2019tilings}, bounded-degree subgraphs~\cite{bottcher2017embedding} and powers of Hamilton cycles~\cite{antoniuk2020high,antoniuk,BPSS_square,dudek2020powers,nenadov2018sprinkling}).
    
    Another flourishing trend is to investigate the emergence of rainbow structures in uniformly edge-coloured graphs.
    Given an edge-coloured graph \(G\), a subgraph \(H\) of \(G\) is \emph{rainbow} if each edge of \(H\) has a different colour.
    Moreover we say that a graph \(G\) is \emph{uniformly edge-coloured} in a set of colours \(\cols\) if each edge  of \(G\) gets a colour independently and uniformly at random from \(\cols\).
    Instances of this problem in the random graph setting can be found in~\cite{aigner-horev-trees,frieze-bal-optimal-colours,ferber-optimal-colours,ferber-kriv,frieze-loh}.
    Here we focus on the perturbed setting.

    Let $H=H(n)$ be a given $n$-vertex graph, $\delta \in (0,1)$ and $\cC$ be a set of colours.
    Suppose we would like to find (with high probability) a rainbow copy of $H$ in a uniformly coloured perturbed graph $G \sim G_0 \cup \gnp{n}{p}$, with $G_0 \in \cG_{\delta, n}$ and colours in $\cC$.
	In particular, $G$ must contain a copy of $H$ and thus $p$ needs to satisfy $p \ge C\hat{p}$ where $\hat{p}$ is a perturbed threshold at $\delta$ for the containment of $H$, and $C$ is a large enough constant.
    Moreover, because in a rainbow copy each edge gets a different colour, the colour set $\cC$ must satisfy $|\cC| \ge e(H)$.
	Our following result implies that, in many cases, these two conditions are asymptotically enough to guarantee a rainbow $H$ with high probability.

    \begin{theorem} \label{thm:mc_diarmid_argument}
        Let $p, \eps \in (0,1)$, set $\mu := \frac{\eps(1-p)-p}{(1+\eps)(1-p)}$ and $q:=(1+\eps^{-1})p$, and suppose that $\mu > 0$.
        Let $\calH$ be a collection of subgraphs of $K_n$, each with $m$ edges, and $G_0$ be an $n$-vertex graph.
        Let \(G_0'\) be the random subgraph of \(G_0\), where each edge is sampled independently with probability $\mu$.
        Then
        \[
            \PP\left[
            \begin{array}{c}
            G_0' \cup \text{\(\gnp{n}{p}\) contains} \\
            \text{some \(H \in \calH\)}
            \end{array} \right] \le \PP\left[
            \begin{array}{c}
    		\text{a uniformly edge-coloured \(G_0 \cup \gnp{n}{q}\),} \\
            \text{with colours in $[(1+\eps)m]$,
            contains a rainbow \(H \in \calH\)}
            \end{array}
            \right] .
    	\]
    \end{theorem}
	The proof of~\Cref{thm:mc_diarmid_argument} builds upon an ingenious coupling idea of McDiarmid~\cite{Mcdiarmid}, which has already been used for rainbow problems in random graphs by Ferber and Krivelevich~\cite{ferber-kriv}. 
    Our result is extremely versatile and provides a general machinery to translate existence results into rainbow ones.
    Indeed, let $p=o(1)$ and fix $\eps \in (0,1)$ as in the statement of the theorem.
    Let $\delta \in (0,1)$ and suppose that $G_0 \in \cG_{n,\delta}$.
	Then $\delta(G_0') \ge \eps \delta n / 2$ with high probability.
	Now observe that if $p$ is large enough so that, for every $G_1 \in \cG_{n,\eps\delta/2}$, the perturbed graph $G_1 \cup \gnp{n}{p}$ contains a copy of $H$ with high probability, then \Cref{thm:mc_diarmid_argument} implies the following: for every $G_0 \in \cG_{n,\delta}$, if $G_0 \cup \gnp{n}{(1+\eps^{-1})p}$ is uniformly edge-coloured in $[(1+\eps)e(H)]$, then with high probability it contains a rainbow copy of $H$.
    I.e.\ a perturbed threshold at $\eps \delta/2$ for containing $H$ provides an upper bound on the `rainbow perturbed threshold' at $\delta$ for containing a rainbow $H$ when we are allowed to use $(1+\eps)e(H)$ colours.

    Note that in general we cannot conclude that this gives the optimal edge probability for a rainbow $H$. 
    For example, consider the case of $H$ being a $K_3$-factor.
    Then $n^{-2/3}$ is a perturbed threshold for $0<\delta<1/3$ (see~\cite{balogh2019tilings}), $\log n/n$ is a perturbed threshold for $\delta=1/3$ (see~\cite{BPSS_triangles}) and $n^{-1}$ is a perturbed threshold for $1/3 < \delta < 2/3$ (see~\cite{han2019tilings}), while by the Corr\'adi-Hajnal Theorem no random edges are needed when $\delta \ge 2/3$.
    In particular, observe that the perturbed threshold has a `jump' at $\delta=1/3$.
    Therefore, while the existence threshold at $1/3$ is $\log n/n$, \Cref{thm:mc_diarmid_argument} needs $p \ge C n^{-2/3}$ to guarantee a rainbow $H$ with high probability when $\delta = 1/3$, where $C$ is a large enough constant.
    Indeed, using the same notation as in \Cref{thm:mc_diarmid_argument}, we can find graphs $G_0 \in \cG_{n,1/3}$ such that with high probability the minimum degree of $G_0'$ drops strictly below $n/3$. Then, by the discussion above, for $G_0' \cup \gnp{n}{p}$ to contain a triangle factor, we need $p \ge Cn^{-2/3}$.
    
    However, if we have a function $\hat{p}$ which is a perturbed threshold for all $\delta \in (0,\delta_0)$, then \Cref{thm:mc_diarmid_argument} implies that $\hat{p}$ is also a rainbow perturbed threshold at every $\delta \in (0,\delta_0)$ when colouring with $(1+o(1))e(H)$ colours. 
    
    For example, this is the case for rainbow trees.
    Krivelevich, Kwan and Sudakov~\cite{kriv-kwan-sudakov} proved that for any $\delta \in (0,1/2)$ the function $n^{-1}$ is a perturbed threshold for containing a given spanning bounded-degree tree.
    Thus \Cref{thm:mc_diarmid_argument} immediately implies the following.
	\begin{corollary}
    \label{cor:rainbow:trees}
		For any $\eps, \delta \in (0,1)$ and $d \ge 2$ an integer, there exists $C=C(\eps,\delta,d)$ such that the following holds.
		Let $T$ be an $n$-vertex tree with maximum degree at most $d$. Then a uniformly coloured $G_0 \cup \gnp{n}{C/n}$ with $G_0 \in \cG_{\delta,n}$ and colours in $[(1+\eps)n]$ admits a rainbow copy of $T$, with high probability.
	\end{corollary}
    Observe that, because of the result from~\cite{kriv-kwan-sudakov} cited above, \Cref{cor:rainbow:trees} has the optimal edge probability (up to a constant factor).
    Moreover, it improves upon a result of Aigner-Horev, Hefetz and Lahiri~\cite{aigner-horev-trees}, who proved the same conclusion with the $C/n$-term in the probability replaced by $\omega(1)/n$, and confirms their conjecture that $C/n$ is already enough.   

    Similarly, \Cref{thm:mc_diarmid_argument} has consequences for rainbow Hamilton cycles.
    Indeed, it implies that for any $\eps, \delta \in (0,1)$, there exists $C=C(\eps,\delta)$ such that for any $G_0 \in \cG_{\delta,n}$ we have that a uniformly coloured $G_0 \cup \gnp{n}{C/n}$ with colours in $[(1+\eps)n]$ admits a rainbow Hamilton cycle with high probability.
    We remark that this was already proved by Aigner-Horev and Hefetz \cite{aigner-horev_cycle} using different, ad hoc methods.

    \Cref{thm:mc_diarmid_argument} still leaves open the question if the extra colours are needed. Namely, are $e(H)$ colours enough to guarantee a rainbow $H$ with high probability?
    This seems to be a much more challenging problem and, even for specific choices of $H$, not much is known.
    To the best of our knowledge, the only known rainbow result in the perturbed graph setting with the exact number of colours is offered by~\cite{KLS}, where we show that a uniformly coloured $G_0 \cup \gnp{n}{C/n}$ with colours in $[n]$ contains a rainbow Hamilton cycle with high probability (in fact, we prove a version of this for directed graphs).
    This is clearly best possible both in terms of the edge probability (up to a constant factor, from the result of~\cite{bohman} cited above) and the number of colours (since a Hamilton cycle has $n$ edges).
    Here we pursue this direction further and give an exact result for the containment of rainbow bounded-degree trees.
    
    \begin{theorem}
    \label{thm:main}
        Let $\delta \in (0,1)$ and let $d \ge 2$ be a positive integer. Then there exists $C=C(\delta,d) > 0$ such that the following holds.
        Let $T$ be a tree on $n$ vertices with maximum degree at most $d$,
        let $G_0$ be a graph on $n$ vertices with minimum degree at least $\delta n$ and suppose
        \(G \sim G_0\cup \gnp{n}{C/n}\) is uniformly coloured in \([n-1]\). 
        Then,
        with high probability, \(G\) contains a rainbow copy of \(T\).
    \end{theorem}
    
    \Cref{thm:main} provides a rainbow variant of the result of Krivelevich, Kwan and Sudakov~\cite{kriv-kwan-sudakov} cited above.
    Observe that \Cref{thm:main} has the optimal number of colours, and  also has the optimal edge probability (up to a constant factor).
    For comparison, this improves upon the result of Aigner-Horev, Hefetz and Lahiri~\cite[Theorem 1.3]{aigner-horev-trees}, who required edge probability \(\omega(n^{-1})\) and \(\eps n\) additional colours to get the same conclusion.

    An important step of our proof of \Cref{thm:main} relies on the following theorem which allows to embed in a rainbow fashion an almost-spanning bounded-degree tree in a uniformly coloured $\gnp{n}{p}$ provided $p \ge C/n$ for large enough $C$. 
    \begin{theorem} 
    \label{thm:intro-rainbow-almost-spanning-embedding}
        Let \(\eps \in (0,1)\) and let \(d \ge 2\) be a positive integer.
        Then there exists \(C=C(\eps,d)>0\) such that the following holds.
        Let $T$ be a tree on at most $(1 - \eps)n$ vertices with maximum degree $d$ and suppose that $\rg \sim \gnp{n}{C/n}$ is coloured uniformly in $[n]$. Then, with high probability, \(\rg\) contains a rainbow copy of \(T\).
    \end{theorem}
    
    Aigner-Horev, Hefetz and Lahiri~\cite{aigner-horev-trees} already proved that the same conclusion holds when $C/n$ is replaced by $\omega(1)/n$.
    They conjectured that $\omega(1)/n$ can be replaced by $C/n$ and thus \Cref{thm:intro-rainbow-almost-spanning-embedding} resolves their conjecture.
    We remark that \Cref{thm:intro-rainbow-almost-spanning-embedding} is an immediate consequence of two previous results: the uncoloured version of \Cref{thm:intro-rainbow-almost-spanning-embedding} proved by Alon, Krivelevich and Sudakov (c.f.~Theorem~1.1 in~\cite{AKS}) and a general tool of Ferber and Krivelevich~\cite{ferber-kriv} (c.f.~\Cref{thm:ferber-kriv-conseq}) which allows to translate uncoloured results into rainbow ones.
    However, for our purposes, \Cref{thm:intro-rainbow-almost-spanning-embedding} will not be enough and we will need a more general version, namely when $\rg$ is a random subgraph of a pseudorandom graph.
    For a precise definition of what we mean by pseudorandom and why this generalisation is needed, we refer the reader to \Cref{sec:almost-spanning}.
    
	\paragraph{Organisation.}
		The rest of the paper is organised as follows.
		In \Cref{sec:mc_diarmid} we prove \Cref{thm:mc_diarmid_argument} and in \Cref{sec:almost-spanning} we prove a more general version of \Cref{thm:intro-rainbow-almost-spanning-embedding} (c.f.~\Cref{thm:rainbow-almost-spanning-embedding}).
		\Cref{sec:overview} provides an outline of our arguments for \Cref{thm:main}, together with the tools and auxiliary lemmas we use in its proof.
		The proof splits into two cases, according to the structure of the tree $T$ we wish to embed: when $T$ has many leaves (c.f.~\Cref{thm:many_leaves}, proved in \Cref{sec:many_leaves}) and when $T$ has many bare paths (c.f.~\Cref{thm:many_paths}, proved in \Cref{sec:many_paths}).
        We then finish by some concluding remarks in \Cref{sec:concluding}.
		One supplementary proof is moved to \Cref{appendix}.
    
	\paragraph{Notation.}        
        Given a graph $G$, a vertex $v \in V(G)$ and a subset $X \subseteq V(G)$, $E_G(v,X)$ denotes the set of edges of the form $vx$ with $x \in X$, and \(N_G(X)\) denotes the set of vertices with at least one neighbour in $X$.
        A \emph{bare path} in $G$ is a path whose interior vertices have degree $2$ in $G$.
        For a graph $G$ and \(p\in[0,1]\), the \emph{\(p\)-random subgraph} of \(G\), denoted by \(G_p\), is the random graph resulting from sampling each edge of \(G\) independently with probability \(p\).

        A \emph{digraph} $D$ is a set of vertices together with a set of ordered pairs of distinct vertices and the \emph{minimum semi-degree} $\delta^0(D)$ of $D$ is the minimum over in- and out-degrees of vertices in $D$. 
        Moreover \(\dnp{n}{p}\) denotes the binomial random digraph on \(n\) vertices, that is the digraph on $[n]$ where each ordered pair of distinct vertices forms a directed edge independently with probability \(p\).

		Given an edge-coloured graph $G$, we denote the colour of an edge $e$ by $\cols(e)$ and the set of colours on the edges of a subgraph \(G'\) by \(\cols(G')\).
		Moreover we say that \(G'\) is \emph{spanning} in a colour set \(\cols'\) if \(\cols(G') = \cols'\).
        
        For $a$, $b$, $c \in (0, 1]$, we write, for example, $a \ll b \ll c$ in our statements to mean that there are increasing functions $f, g : (0, 1] \to (0, 1]$ such that whenever $a \le f (b)$ and $b \le g(c)$, the subsequent statement holds.

        Throughout, \(\log n\) denotes the natural logarithm.

\section{McDiarmid argument for randomly perturbed graphs} 
\label{sec:mc_diarmid}
	As alluded to in the introduction, the proof of \Cref{thm:mc_diarmid_argument}, relating the probability of finding a rainbow $H$ in a perturbed graph to that of finding a copy of $H$ in an uncoloured perturbed graph, is based on a coupling argument. This is inspired by a result of Ferber and Krivelevich \cite[Theorem 1.2]{ferber-kriv}, which in turn uses a coupling trick due to McDiarmid \cite{Mcdiarmid}.
    \begin{proof}[Proof of \Cref{thm:mc_diarmid_argument}]
        Let $\cC:=[(1+\eps)m]$ be the palette of colours.
        We define a sequence of graphs $\Gamma_0, \dots, \Gamma_N$, where $N := \binom{n}{2}$ and each graph is equipped with a colouring of its edges, where we allow an edge to be coloured with multiple colours.
        
        Let $e_1, \dots, e_N$ be an arbitrary enumeration of all the edges of $K_n$.
        For $0 \le i \le N$ define $\Gamma_i$ as follows.
		\begin{itemize}
			\item
				For $1 \le j \le i$,
				\begin{itemize}
					\item
						If $e_j \in E(G_0)$, add $e_j$ to $\Gamma_i$ and assign it a colour uniformly at random from $\cC$. 
					\item
						If $e_j \not\in E(G_0)$, add $e_j$ to $\Gamma_i$ with probability $q$ and assign it a colour uniformly at random from $\cC$.
				\end{itemize}
			\item
				For $j > i$,
				\begin{itemize}
					\item
						If $e_j \in E(G_0)$, then add $e_j$ to $\Gamma_i$ with probability $\eps/(1+\eps)$ and assign it all colours from $\cC$. 
					\item
						If $e_j \not\in E(G_0)$, then add $e_j$ to $\Gamma_i$ with probability $p$ and assign it all colours from $\cC$.
				\end{itemize}
		\end{itemize}
        We remark that all the random choices mentioned above are mutually independent.

		We claim that $\Gamma_0$ is distributed as $G_0' \cup \gnp{n}{p}$, with edges assigned all colours in $\cC$. Indeed, we have 
		\begin{equation*}
			\PP\left[{e \in E\left(G_0' \cup \gnp{n}{p}\right)}\right] = \left\{
				\begin{array}{ll}
					\mu + (1-\mu)p = \frac{\eps}{1+\eps}, & \text{if } e \in E(G_0)\\
					p, & \text{if } e \notin E(G_0)
				\end{array}
			\right.,
		\end{equation*}
		which is exactly the probability that $e \in \Gamma_0$, as claimed.
		Because all edges in $\Gamma_0$ have all colours in $\cC$, it is therefore the case that
		\begin{equation*}
			\PP\left[\text{$G_0' \cup \gnp{n}{p}$ has a copy of some $H \in \calH$}\right]
			= \PP\left[\text{$\Gamma_0$ has a rainbow copy of some $H \in \calH$}\right].
		\end{equation*}
		We also have
		\begin{equation*}
			\PP\left[
				\begin{array}{c}
					\text{$G_0 \cup \gnp{n}{q}$ has} \\ 
					\text{a rainbow copy of some $H \in \calH$}
				\end{array}
				\right]
			= \PP\left[
				\begin{array}{c}
					\text{$\Gamma_N$ has} \\ 
					\text{a rainbow copy of some $H \in \calH$}
				\end{array}
				\right],
		\end{equation*}
		since $\Gamma_N$ is distributed as $G_0 \cup \gnp{n}{q}$ with edges coloured uniformly in $\cC$.
        Therefore, in order to complete the proof, it is enough to show
		\begin{equation} \label{eqn:rainbow}
            \PP\left[
            \begin{array}{c}
            \Gamma_{i-1} \text{ contains} \\
            \text{a rainbow copy of some \(H \in \calH\)}
            \end{array} \right] \le \PP\left[
            \begin{array}{c}
            \Gamma_i \text{ contains } \\
    		\text{a rainbow copy of some $H \in \calH$}
            \end{array}
            \right] ,
		\end{equation}
        for each $i \in [N]$.
        Observe that there are three mutually exclusive scenarios:
		\begin{enumerate}[label = \rm(\alph*)]
            \item \label{mcdiamird_i} $\Gamma_{i-1}$ contains a rainbow copy of some $H \in \calH$ not using $e_i$;
            \item \label{mcdiamird_ii} $\Gamma_{i-1}$ does not contain a rainbow copy of any $H \in \calH$, not even if we add $e_i$ and assign it all colours;
            \item \label{mcdiamird_iii} $\Gamma_{i-1}$ contains a rainbow copy of some $H \in \calH$ if we add $e_i$ and assign it all colours, but does not contain a rainbow copy of some $H \in \calH$ that avoids $e_i$.
        \end{enumerate}
    
		To prove \eqref{eqn:rainbow}, it suffices to show that the inequality holds when conditioning each side on each of \ref{mcdiamird_i}, \ref{mcdiamird_ii} and \ref{mcdiamird_iii}. This holds with equality if \ref{mcdiamird_i} or \ref{mcdiamird_ii} holds: if \ref{mcdiamird_i} holds, both sides are $1$, and if \ref{mcdiamird_ii} holds, both sides are $0$.
		Now consider \ref{mcdiamird_iii}.
        \[
            \PP\left[
            \begin{array}{c}
            \Gamma_{i-1} \text{ contains} \\
            \text{a rainbow copy of some \(H \in \calH\)}
            \end{array} \bigg| \ref{mcdiamird_iii} \right]  =
            \begin{cases}
                \eps/(1+\eps), & \text{if $e_i \in E(G_0)$} \\
                p, & \text{if $e_i \not\in E(G_0)$} 
            \end{cases}\, .
    	\]
        The crucial observation is that if $e_i$ can complete a rainbow copy of some $H \in \calH$ in $\Gamma_{i-1}$ then there are at least $\eps m$ colours (out of the $(1+\eps)m$ in the palette) for $e_i$ which yields a rainbow copy of $H$ in $\Gamma_i$. 
        Therefore
        \[
            \PP\left[
            \begin{array}{c}
            \Gamma_{i} \text{ contains} \\
            \text{a rainbow copy of some \(H \in \calH\)}
            \end{array} \bigg| \ref{mcdiamird_iii} \right]  \ge
            \begin{cases}
                \eps/(1+\eps), & \text{if $e_i \in E(G_0)$} \\
                q \eps/(1+\eps) = p, & \text{if $e_i \not\in E(G_0)$} 
            \end{cases}\, .
    	\]
		Thus \eqref{eqn:rainbow} holds, completing the proof.
    \end{proof}

\section{Almost spanning rainbow trees in random graphs} 
	\label{sec:almost-spanning}

	In order to prove \Cref{thm:main}, about finding a rainbow copy of a spanning bounded-degree tree in a perturbed graph, we have to embed a rainbow copy of a fixed bounded-degree spanning tree $T$ in a uniformly coloured $G_0 \cup \gnp{n}{p}$, with colours in $[n-1]$, using both edges of $G_0$ and random edges of $\gnp{n}{p}$.
	We will do so by first embedding a certain almost-spanning subtree of $T$ in $\gnp{n}{p}$ only (in a rainbow fashion) and then completing it to a (rainbow) embedding of $T$ in the full perturbed graph $G_0 \cup \gnp{n}{p}$.
	Therefore, we first show that we can embed almost-spanning trees with bounded degree in uniformly coloured random graphs in a rainbow fashion.

	However, rather than standard random graphs, we consider random subgraphs of pseudorandom graphs, for the following reason.
	To prove \Cref{thm:main} when the tree has many bare paths (c.f.\ \Cref{thm:many_paths}), we will first build an absorbing structure, using edges of both $G_0$ and $\gnp{n}{p}$, and then embed a rainbow almost-spanning forest in the remainder, using only edges from $\gnp{n}{p}$.
	In order to guarantee that the colouring is random for the second step as well, we use the following trick.
	We partition the edges of \(K_n\) randomly into two sets.
	Then for the first step we only use the edges of $G_0 \cup \gnp{n}{p}$ which appear in the first set and, similarly, for the second step, we only use those spanned by $\gnp{n}{p}$ among the second set of edges.
	Hence, the latter is not a random subgraph of the complete graph, but a random subgraph of a graph that has pseudorandom properties with high probability.
	We will use the following definition of pseudorandomness, where the number of edges between any two not-too-small disjoint vertex sets is close the expected number of such edges in $\gnp{n}{1/2}$.

	\begin{definition}[Pseudorandom graph]\label{def:pseudorandom}
		A graph \(G\) on \(n\) vertices is \emph{pseudorandom} if for any two disjoint subsets of vertices \(A,B \) with \(\abs{A} \cdot \abs{B} \ge 250 n\), we have
		\(e(A,B) \ge \abs{A} \abs{B} /3  \).
	\end{definition}

	We can now state the main result of this section.

	\begin{theorem} 
	\label{thm:rainbow-almost-spanning-embedding}
		Let \(1/C \ll \eps, 1/d <1\) with $d \ge 2$ being a positive integer.
		Let $T$ be a tree on at most $(1 - \eps)n$ vertices with maximum degree $d$, let $G$ be a pseudorandom graph on $n$ vertices, and write $p := C/n$.
		Suppose that $G_p$ is coloured uniformly with \(n\) colours. Then, with high probability, \(G_{p}\) contains a rainbow copy of \(T\).
	\end{theorem}

	Observe that \Cref{thm:intro-rainbow-almost-spanning-embedding} is now a simple corollary of \Cref{thm:rainbow-almost-spanning-embedding}, since the complete graph on \(n\) vertices is a pseudorandom graph.
	In order to prove \Cref{thm:rainbow-almost-spanning-embedding}, we first prove an uncoloured version of it.
	\begin{proposition} \label{prop:sample-pseudorandom-tree}
		Let \(1/C \ll \eps, 1/d<1\) with $d \ge 2$ being a positive integer.
		Set \(p:=C/n\) and let \(G\) be a pseudorandom graph on \(n\) vertices.
		Then, with high probability, \(G_p\) contains a copy of every tree with at most $(1-\eps)n$ vertices and maximum degree at most $d$.
	\end{proposition}

	This is a generalisation of a result due to Alon, Krivelevich and Sudakov~\cite[Theorem 1.1]{AKS} who already proved that the same conclusion holds when $G$ is the complete graph on $n$ vertices.
	Their result relies on the fact that sparse and almost-regular `robust expanders' contain a copy of every almost-spanning bounded-degree tree (c.f.~\Cref{thm:AKS_expander_contains_trees}).
	We will employ the same approach and prove \Cref{prop:sample-pseudorandom-tree} using essentially the same arguments, but with minor differences in calculations.
	For the sake of completeness, we give a proof in \Cref{appendix}.

	The next result allows to translate \Cref{prop:sample-pseudorandom-tree} into \Cref{thm:rainbow-almost-spanning-embedding}, and is a simple consequence of a general result of Ferber and Krivelevich~\cite[Theorem 1.2]{ferber-kriv} for binomial random subgraphs of uniformly edge-coloured hypergraphs. 

	\begin{theorem}[Consequence of~{\cite[Theorem 1.2]{ferber-kriv}}] \label{thm:ferber-kriv-conseq}
		Let $\eps, p \in (0,1)$ and set $q := \eps^{-1}p$.
		Suppose that \(\calH\) is a collection of subgraphs of \(K_n\) with at most 
		\((1-\eps)n\) edges.
		Then 
		\[
			\PP\left[
			\begin{array}{c}
			\text{\(\gnp{n}{p}\) contains} \\
			\text{some \(H \in \calH\)}
			\end{array} \right] \le \PP\left[
			\begin{array}{c}
			\text{a uniformly edge-coloured \(\gnp{n}{q}\),} \\
			\text{with colours in \([n]\),
			contains a rainbow \(H \in \calH\)}
			\end{array}
			\right] .
		\]
	\end{theorem}

	It is now easy to prove \Cref{thm:rainbow-almost-spanning-embedding}.

	\begin{proof}[Proof of \Cref{thm:rainbow-almost-spanning-embedding}]
		Let $T$ be a tree on at most $(1-\eps)n$ vertices with maximum degree $d$ and $G$ be an $n$-vertex pseudorandom graph on $V$.
		Let \(C_0\) be given by \Cref{prop:sample-pseudorandom-tree} on input $\eps$ and $d$, and set $p_0:=C_0/n$.
		Let \(\calH\) be the collection of labelled copies of \(T\) in \(G\). 

		Observe that $G_{p_0}$ contains a copy of $T$ with high probability, by \Cref{prop:sample-pseudorandom-tree}.
		Let \(\gnp{n}{p_0}\) be the binomial random graph on $V$, coupled with \(G_{p_0}\) so that \(G_{p_0} \subseteq \gnp{n}{p_0}\).
		Then it follows that with high probability 
		\(\gnp{n}{p_0}\) contains a graph in \(\calH\).

		Set \(C := \eps^{-1} C_0\) and \(p:=C/n\).
		Let \(\gnp{n}{p}\) be the binomial random graph on $V$, coupled with \(G_p\) so that an edge $e$ in $G$ is in $G_p$ if and only if it is in $\gnp{n}{p}$, and colour each of its edges uniformly in \([n]\).
		Then \Cref{thm:ferber-kriv-conseq} implies that with high probability \(\gnp{n}{p}\) contains a rainbow \(H\in \calH\) and, because \(\calH\) is a collection of subgraphs of \(G\) and by the coupling, it follows that so does \(G_p\).
		That is, with high probability, \(G_p\) contains a rainbow copy of \(T\).
	\end{proof}

\section{Overview of \Cref{thm:main}}  \label{sec:overview}
	
	Let $G_0$ be a graph on vertex set $[n]$ with minimum degree at least $\delta n$.
    Let \(G\sim G_0 \cup \gnp{n}{C/n}\) and suppose \(G\) is uniformly coloured in \([n-1]\).
    Let \(T\) be an $n$-vertex tree with maximum degree at most \(d\) that we wish to embed in a rainbow fashion in \(G\).
    To aid our embedding, we seek simple structures in $T$, for which we use the following observation of Krivelevich~\cite{kriv-trees}, where we recall that a \emph{bare path} in $T$ is a path whose interior vertices have degree $2$ in $T$.
    \begin{lemma}[{\cite[Lemma 2.1]{kriv-trees}}]
    \label{lemma:leaves_or_paths}
        For any integers $n, k > 2$, a tree with $n$ vertices either has at least $n/4k$ leaves or a collection of at least $n/4k$ vertex-disjoint bare paths, each of length $k$.
    \end{lemma}
    Our proof splits into two cases, according to the structure of the tree $T$.
    \begin{theorem}[Trees with many leaves]
    \label{thm:many_leaves}
        Let $1/C \ll \zeta, \delta, 1/d<1$ with $d \ge 2$ an integer and $\delta \in (0,1)$.
        Let $T$ be a tree on $n$ vertices with maximum degree at most $d$, containing at least $\zeta n$ leaves.
        Let $G_0$ be an $n$-vertex graph with minimum degree at least $\delta n$, and suppose that \(G \sim G_0\cup \gnp{n}{C/n}\) is uniformly coloured in \([n-1]\). 
        Then,
        with high probability, \(G\) contains a rainbow copy of \(T\).
    \end{theorem}
    \begin{theorem}[Trees with many bare paths]
    \label{thm:many_paths}
        Let $1/C \ll \zeta \ll \delta, 1/d<1$ with $d \ge 2$ and  $\zeta n / 24$ being integers and $\delta \in (0,1)$.
        Let $T$ be a tree on $n$ vertices with maximum degree at most $d$, containing at least $\zeta n/24$ bare paths, each with length $6/\zeta$.
        Let $G_0$ be an $n$-vertex graph with minimum degree at least $\delta n$, and suppose that
        \(G \sim G_0\cup \gnp{n}{C/n}\) is uniformly coloured in \([n-1]\). 
        Then,
        with high probability, \(G\) contains a rainbow copy of \(T\).
    \end{theorem}

    \Cref{thm:main} easily follows by combining \Cref{lemma:leaves_or_paths} and \Cref{thm:many_leaves,thm:many_paths}.
    For both the two theorems above, we employ the following strategy.
    We remove the paths or leaves from $T$ and embed the remaining almost-spanning forest in a rainbow fashion using \Cref{thm:rainbow-almost-spanning-embedding}.
    The challenge is then to embed the deleted paths or leaves covering exactly the remaining vertices and using exactly the remaining colours.
    We discuss this informally now, mentioning all auxiliary lemmas that we will need, and postpone the precise proofs to subsequent sections.
    Besides the lemmas below, the only other tool we shall need is Chernoff's bound.

    \begin{lemma}[{Chernoff Bound,~\cite[Theorem 2.8]{jlr}}]
		Let \(X\) be the sum of mutually independent indicator random variables.
        Then for any $\delta \in (0,1)$ we have
		\[
		      \PP\Big[\big|X- \EE[X]\big| \ge \delta \cdot \EE[X]\Big] \le 2\exp\left(-\frac{\delta^2}{3} \cdot \EE[X]\right)\, .
		\]
	\end{lemma}
    
    \subsection{Embedding trees with many leaves}
		Suppose that \(T\) has \(\Omega(n)\) leaves and let $L$ be a maximal collection of leaves with \emph{distinct} parents \(M\).
		By the maximum degree assumption, we have $|L| = \Omega(n)$.
		Let $V$ be the vertex set of the perturbed graph \(G\).
		
		We first embed $T\setminus L$ in a rainbow fashion in \(\gnp{n}{C/n}\) using~\Cref{thm:intro-rainbow-almost-spanning-embedding}.
		Completing this to a rainbow embedding of \(T\) essentially amounts to finding a rainbow perfect matching between the image of \(M\) and the uncovered vertices of $V$, which uses all the unused colours.
		This can be reduced to finding a rainbow directed Hamilton cycle in a suitable auxiliary edge-coloured perturbed directed graph.
		For that we will apply
		the following result of the authors, where we recall that \(\dnp{n}{p}\) denotes the binomial random digraph on \(n\) vertices with edge probability \(p\). 

		\begin{theorem}[{\cite[Theorem 1.2]{KLS}}]
		\label{thm:directed_ham_cycle_perturbed_graph}
			Let \(1/C \ll \delta < 1\) and \(D_0\) be a directed graph on vertex set $[n]$ with minimum in- and out-degree at least \(\delta n\).
			Suppose \(D \sim D_0 \cup \dnp{n}{C/n}\) is uniformly coloured in \([n]\).
			Then, with high probability, $D$ has a rainbow directed Hamilton cycle.
		\end{theorem}
    
	\subsection{Embedding trees with many bare paths}
		Suppose now that $T$ has $\Omega(n)$ not-too-short disjoint bare paths.
		Consider $r$ such paths of length $\ell$ (where $r = \Omega(n)$ and $\ell$ is a constant which is not too small), and denote the ends of the $i$-th path by $s_i$ and $t_i$.
		Let \(F\) be the forest resulting from removing the interior vertices of these bare paths from \(T\).
		
		We will use~\Cref{thm:rainbow-almost-spanning-embedding} to embed \(F\) in \(G\).
		However, in order to be able to turn this into a rainbow embedding of \(T\) (by embedding a rainbow collection of  \(r\) paths of length \(\ell\), with the $i$-th path having the images of \(s_i\) and \(t_i\) as endpoints), we first prepare an absorbing structure.
		We remark that this is the reason why \Cref{thm:rainbow-almost-spanning-embedding} is stated for random subgraphs of pseudorandom graphs rather than $\gnp{n}{p}$ directly.

		We will state here several lemmas (namely \Cref{lem:flexible_sets,lem:gadget_exists,lem:connecting_gadgets}) from a manuscript \cite{KL} by the first two authors, where they proved an undirected version of \Cref{thm:directed_ham_cycle_perturbed_graph}. These lemmas all have analogues in \cite{KLS} for the directed setting, and in most cases have very similar proofs, but the directed versions do not immediately imply the undirected ones (due to parallel directed edges $xy$ and $yx$ being coloured independently).
		
		\paragraph{Absorber.}
			Before building our absorber, we set aside a set of flexible vertices and flexible colours, where flexible here refers to the fact that they can be used to connect arbitrary pairs of vertices into short rainbow paths using an arbitrary colour.

		\begin{lemma}[Lemma 6.1 in \cite{KL}]
		\label{lem:flexible_sets}
			Let $1/C \ll \nu \ll \mu \ll \delta < 1$.
			Let $G_0$ be an $n$-vertex graph on $V$ with minimum degree at least $\delta n$ and $G \sim G_0 \cup \gnp{n}{C/n}$ be uniformly coloured in $\cC:=[n-1]$.
			Then there exist $V_{\flex} \subseteq V$ and $\cC_{\flex} \subseteq \cC$ of size $\mu n$ such that with high probability the following holds.
			For all $u,v \in V$, $c\in \cC$, and $V'_{\flex} \subseteq V_{\flex}$ and $\cC_{\flex}' \subseteq \cC_{\flex}$ of size at least $(\mu - \nu)n$,
			there exists a rainbow path of length seven with endpoints \(u,v\), internal vertices in \(V'_{\flex}\) and colours in \(\cC_{\flex}' \cup \{c\}\), that contains the colour \(c\).
		\end{lemma}
		
		The building block of our absorber is given by the so-called 
		\emph{\((v,c)\)-gadget}.
		These have been introduced by Gould, Kelly, K\"uhn and Osthus~\cite{kuhn-osthus} in the context of random optimal proper colourings of the complete graph, and have already been used for perturbed graphs in~\cite{KLS}.

		\begin{definition}[Gadget] \label{def:gadget}
			Let \(v\) be a vertex and \(c\) a colour.
			A \emph{\((v,c)\)-gadget}, denoted by $A_{v,c}$,
			is the edge-coloured graph on $11$ vertices depicted in \Cref{fig:absorber}.
			
			With reference to the notation in \Cref{fig:absorber}, we call \(P:=u_1vu_2P_1w_2w_3P_2w_1w_4\) the \emph{\((v,c)\)-absorbing path}
			and \(P':=u_1u_2P_1w_2w_1P_2w_3w_4\) the \emph{\((v,c)\)-avoiding path}.    
			Moreover, we call $u_1$ the
			\emph{first vertex} of the absorber, and $w_4$ the \emph{last vertex}.
			Finally, we say that \(V(A_{v,c})\setminus \{v\}\) are the \emph{internal vertices} of \(A_{v,c}\) and \(\cols(A_{v,c}) \setminus \{c\}\) are the \emph{internal colours}.
		\end{definition}

		Observe that $P$ and $P'$ in the definition of a gadget are both rainbow paths and share the same endpoints, which are the first and last vertex of the absorber.
		Moreover, $P$ is spanning in $V(A_{v,c})$ and $\cC(A_{v,c})$ and, similarly, for \(P'\) we have \( V(P') = V(P)\setminus \{v\} =  V(A_{v,c})\setminus \{v\} \)
		and
		\(
		\cols(P')
		= \cols(P)\setminus\{c\}
		= \cols(A_{v,c})\setminus \{c\}
		\).

		\begin{figure}[H] 
		\centering
		\begin{subfigure}{0.6\linewidth}
				\includegraphics[width=\linewidth]{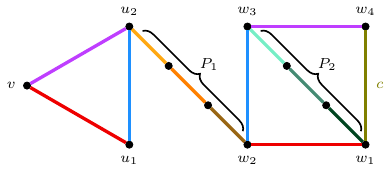}
			\end{subfigure}%
		\vskip 5pt
		\begin{subfigure}[b]{0.45\linewidth}
			\includegraphics[width=\linewidth]{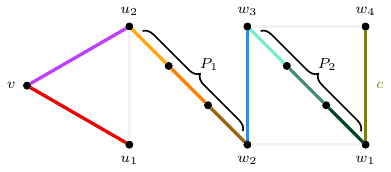}
			\end{subfigure}
		\hfill
		\begin{subfigure}[b]{0.45\linewidth}
		\includegraphics[width=\linewidth]{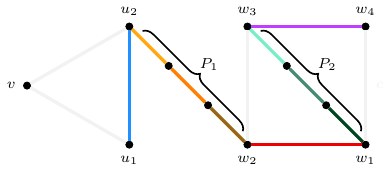}
			\end{subfigure}%
			\caption{
			The top figure shows a \((v,c)\)-gadget \(A_{v,c}\), where we remark that $\cols(w_1w_4)=c$ and all colours are pairwise different, except for $\cols(vu_1)=\cols(w_1w_2)$, $\cols(vu_2)=\cols(w_3w_4)$ and $\cols(u_1u_2)=\cols(w_2w_3)$. The bottom-left figure highlights the absorbing path $P$ and the bottom-right one highlights the avoiding path $P'$.
			}\label{fig:absorber}
		\end{figure}
		
		The existence of $(v,c)$-gadgets is guaranteed by the following lemma.
		
		\begin{lemma}[Lemma 5.2 in \cite{KL}]
		\label{lem:gadget_exists}
			Let $1/C \ll \nu \ll \delta < 1$.
			Let $G_0$ be an $n$-vertex graph on $V$ with minimum degree at least $\delta n$ and $G \sim G_0 \cup \gnp{n}{C/n}$ be uniformly coloured in $\cC:=[n-1]$.
			Then with high probability the following holds.  
			For any $v \in V$ and $c \in \cC$ and for all $V' \subseteq V$ and $\cC' \subseteq \cC$ that have size at least $(1-\nu)n$, there exists a $(v,c)$-gadget with internal vertices in $V'$ and internal colours in $\cC'$.
		\end{lemma}
		
		\Cref{lem:gadget_exists} allows to find many vertex- and colour-disjoint gadgets. 
		In order to build a system of paths with a global absorbing property, we connect several of them.
		Suppose, for example, we are given a $(v,c)$-gadget $A_{v,c}$ and a $(v',c')$-gadget $A_{v',c'}$, that are vertex- and colour-disjoint.
		By connecting the last vertex of $A_{v,c}$ to the first vertex of $A_{v',c'}$ with a short rainbow path (vertex- and colour-disjoint of the gadgets), we obtain a structure which can absorb the pairs $(v,c)$ and $(v',c')$ simultaneously.
		The existence of such short rainbow paths is guaranteed by the following lemma.

		\begin{lemma}[Corollary of Lemma 4.3 in \cite{KL}]
		\label{lem:connecting_gadgets}
			Let $1/C \ll \rho, \lambda \ll \delta$.        
			Let $G_0$ be an $n$-vertex graph on $V$ with minimum degree at least $\delta n$ and $\cC$ be a set of colours of size $n-1$.
			Let $G \sim G_0 \cup \gnp{n}{C/n}$ be uniformly coloured in $\cC$. 
			Then with high probability the following holds. 
			For all subsets $V' \subseteq V$ and $\cC' \subseteq \cC$ of size at least $(1 - \nu)n$ and any distinct $u, v \in V$, there exists a rainbow path of length three with $u$ and $v$ as endpoints, with internal vertices in $V'$ and colours in $\cC'$.
		\end{lemma}

		\paragraph{Template.} 
			Note that we only have enough space to accommodate $O(n)$ gadgets.
			The way we choose which pairs $(v,c)$ to absorb (in order for the final structure to have strong absorbing properties) will be dictated by an auxiliary \emph{template} graph.
			This technique has been introduced by Montgomery~\cite{montgomery_arxiv,montgomery} and has already found a number applications. 
		
			\begin{lemma}[Variant of Lemma 2.8 in~\cite{montgomery_arxiv}]
				\label{lem:template}
				Let $1/n \ll \zeta \le 1$ and suppose that $\zeta n$ is an integer.
				Then there exists a bipartite graph $H$ on vertex classes $R$ and $S_1 \cup S_2$ with $|R| = (2 - \zeta)n$, $|S_1| = |S_2| = n$ and $d_H(x)=40$ for each $x \in R$, such that the following is true. 
				Given any subset $S_2' \subseteq S_2$ with $|S_2'| = \zeta n$, there is a matching between $R$ and $S_1 \cup (S_2 \setminus S_2')$.
			\end{lemma}
			\Cref{lem:template} can be proved identically to the proof of Lemma 2.8 in~\cite{montgomery_arxiv}.
			We call the graph given by \Cref{lem:template} an \emph{$(n,\zeta)$-template graph} on $(R,S_1 \cup S_2)$. 

\section{Embedding trees with many leaves}
	\label{sec:many_leaves}

	\begin{proof}[Proof of \Cref{thm:many_leaves}]
		Let \(\eps, \lambda>0\) be such that
		\[
			1/C \ll \eps \ll \lambda \ll \zeta, \delta, 1/d\,, 
		\]
		let \(\rg_1 \sim \gnp{n}{C/n}\), so that \(G \sim G_0 \cup \rg_1\), and set $V:=V(G)$ and $\cC:=[n-1]$.

		We embed \(T\) in \(G\) in two steps.
		First, we remove a small linear number of leaves with distinct parents, and embed the resulting almost-spanning tree in a rainbow fashion
		using~\Cref{thm:intro-rainbow-almost-spanning-embedding} and $\rg_1$.
		Then we will find a rainbow perfect matching between the images of the parents and the uncovered vertices of $V$, using all remaining colours.

		Let $L$ be a maximal collection of leaves of $T$ with distinct parents. 
		Since $T$ has at least $\zeta n$ leaves and maximum degree at most $d$, we have $|L| \ge d^{-1} \zeta n \ge \lambda n$ and we pick an arbitrary subset of $L$ of size $\lambda n$, which, abusing notation, we denote by $L$.
		Let $M$ be the collection of parents of the leaves in $L$ and observe that $|L| = |M|$. 
		Finally let $T' = T \setminus L$ and note $T'$ is a tree on $(1-\lambda)n$ vertices.

		Let $R$ be a random subset of $V$ of size $(\lambda-\eps) n$.  
		Then, by Chernoff's bound and the union bound, with high probability,
		\begin{equation} \label{eqn:deg-R}
			\text{every \(v\in V\) satisfies }
			\abs{N_{G_0}(v) \cap R} \ge \frac{1}{2} \cdot \delta \lambda n.
		\end{equation}
		We assume that this holds.

		We claim that $\rg_1[V \setminus R]$ contains a rainbow copy of $T'$, with high probability.
		Writing $n' :=|V \setminus R|= (1 - (\lambda - \eps))n$, let $\cC'$ be a subset of $\cC$ of size $n'$, and let $\rg_1'$ be the subgraph of $\rg_1[V \setminus R]$ consisting of edges coloured $\cC'$. Then $\rg_1'$ is a copy of a random graph $\gnp{n'}{C'/n'}$, where $C' = (|\cC'|/|\cC|) \cdot C \ge C/2$, which is uniformly coloured in $\cC'$. As $|V(T')| = (1 - \lambda)n \le (1 - \eps/2)n'$, \Cref{thm:intro-rainbow-almost-spanning-embedding} implies that $\rg'$ contains a rainbow copy of $T'$, as claimed.

		Assume that a rainbow $T'$ as above exists, and fix an embedding of it in $V \setminus R$.
		Let $\cC_0$ be the set of colours in $\cC$ not used for the embedding of \(T'\), let $M_0$ be the image of $M$ in the embedding, and let $V_0$ be the set of vertices in $V \setminus R$ that are not used in the embedding. Then
		\begin{align*}
			|\cC_0| & = n -1 - \left(\abs{V(T')}-1\right) = \lambda n, \\
			\abs{V_0} & = n - \abs{V(T')} - \abs{R} = \eps n, \\
			|M_0| & = \lambda n.
		\end{align*}

		We claim that, with high probability, 
		\begin{equation} \label{eqn:deg-M}
			\text{every $v \in V$ satisfies }
			|N_{G_0}(v) \cap M_0| \ge \frac{1}{2} \cdot \delta \lambda n.
		\end{equation}
		Indeed, observe that $M_0$ is distributed uniformly at random among all subsets of $V \setminus R$ of size $\lambda n$\footnote{
		Let us give some more formal details to convince the reader that this line of reasoning is valid. 
		Suppose we embed \(T'\) in a random graph with vertex set \(V'\), with \(V' \cap V = \emptyset\) and \(\abs{V'} = n\).
		Choose uniformly at random a bijection \(\pi : V' \rightarrow V\).
		Then the image of \(V(T')\) under \(\pi\) is distributed uniformly at random among all subsets of \(V\) of size \(\abs{V(T')}\).
		}, thus
		the assertion in \eqref{eqn:deg-M} follows from a standard application of Chernoff's bound and the union bound, using $\eps \ll \lambda$.

        We are left to find a rainbow perfect matching between $M_0$ and $V_0 \cup R$ using all colours in \(\cols_0\), and we will do that in two phases, by first finding a rainbow matching saturating $V_0$.
		Write $G_0':=G_0 \setminus \rg_1$ and note that so far we have only revealed colours of the edges of $\rg_1[V \setminus R]$, and thus the colours of the edges of $G_0'[V_0, M_0]$ are yet to be revealed.

		\begin{claim}
			With high probability, there is a rainbow matching in $G_0'[V_0, M_0]$ which saturates $V_0$ and uses colours in $\cC_0$.
		\end{claim}

		\begin{proof}
			For $v \in V_0$, write $X_v$ for the number of colours from $\cC_0$ appearing on edges in $G_0'[\{v\}, M_0]$.
			We claim that, with high probability, $X_v \ge 2\eps n$ for every $v \in V_0$.
			Note that this implies the claim as, since $|V_0|=\eps n$, the required rainbow matching can be constructed greedily.

			To estimate the probability that $X_v < 2\eps n$, note that if this holds then there is a subset $\cC' \subseteq \cC_0$ of size $2\eps n$ such that all edges of $G_0'$ between $v$ and $M_0$ are coloured using colours in $(\cC \setminus \cC_0) \cup \cC'$. Here we use \eqref{eqn:deg-M} and $\frac{|(\cC \setminus \cC_0) \cup \cC'|}{|\cC|} \le 1-(\lambda - 2\eps)$, as well as the easy fact that, with high probability, $|N_{\rg_1}(v) \cap M_0| \le 100 \log n$ for every $v \in V_0$.
			\begin{align*}
				\PP\left[X_v < 2\eps n\right] 
				& \le \binom{\lambda n}{2\eps n} \cdot \big(1 - (\lambda - 2\eps)\big)^{|N_{G_0 \setminus \rg_1}(v) \cap M_0|} \\
				& \le \left(\frac{e\lambda}{2\eps}\right)^{2\eps n} \exp\Big(-(\lambda - 2\eps) \cdot  \frac{\delta \lambda n}{4}\Big) \\
				& \le \exp\left(\left(2\eps \cdot \log(e\lambda/(2\eps)) - \delta \lambda^2/8 \right)n\right) = o(n^{-1}),
			\end{align*}
			using $\eps \ll \lambda, \delta$.
			It follows that $X_v \ge 2\eps n$ for every $v \in V_0$, with high probability.
		\end{proof}

		Let $\cC_1$ be set of colours still available and \(M_1\) the unsaturated vertices in $M_0$.
		Then $|M_1| = |\cC_1| = |R| = (\lambda - \eps)n$.
		Note that, so far, we have not revealed any colours of edges touching $R$, nor edges of $\rg_1$ touching $R$.
		Also, using \eqref{eqn:deg-R}, \eqref{eqn:deg-M} and the fact that $|M_1|=|M_0|-\eps n$, the graph $G_0[M_1, R]$ is a balanced bipartite graph on $2(\lambda - \eps)n$ vertices, with minimum degree at least $\delta \lambda n/2 - \eps n \ge \delta \lambda n/4$.

		We define three random bipartite graphs $H_0, \rH_1, \rH_2$, with bipartition $(M_1, R)$, as follows. 
		Let the edges in $H_0$ be the edges in $G_0[M_1,R]$ that are not in $\rg_1$ and whose colour is in $\cC_1$, let the edges in $\rH_1$ be the edges in $\rg_1[M_1,R]$ that have a colour in $\cC_1$, and include each pair in $M_1 \times R$ in $\rH_2$ with probability $\lambda C/(2n)$, independently.
		
		\begin{claim}
			Fix an outcome of $H_0$. Then $\rH_1$ can be coupled with $\rH_2$ so that $H_0 \cup \rH_2 \subseteq H_0 \cup \rH_1$.
		\end{claim}

		\begin{proof}
			Note that it suffices to prove $\PP\left[e \in E(\rH_1) \,|\, e \notin E(H_0)\right] \ge \frac{\lambda C}{2n}$ for every $e \in M_1 \times R$. To see this, note first that if $e \notin E(G_0)$ then 
			\begin{equation*}
				\PP\left[e \in E(\rH_1) \,|\, e \notin E(H_0)\right]
				= \PP[e \in E(\rH_1)] 
				= \PP[e \in \rg_1] \cdot \PP[\cC(e) \in \cC_1] 
				= \frac{C}{n} \cdot \frac{|\cC_1|}{n-1} \ge \frac{\lambda C}{2n}.
			\end{equation*}
			Now consider $e \in E(G_0)$.
			Then
			\begin{align*}
				\PP\big[e \in E(\rH_1) \,|\, e \notin E(H_0)\big]
				& = \frac{\PP[e \in E(\rH_1) \setminus E(H_0)]}{\PP[e \notin E(H_0)]} \\[.7em]
				& = \frac{\PP[e \in \rg_1 \text{ and }\cC(e) \in \cC_1]}{\PP[e \in E(\rg_1)] + \PP[e \notin E(\rg_1) \text{ and } \cC(e) \notin \cC_1]} \\[.7em]
				& = \frac{\frac{C}{n} \cdot \frac{|\cC_1|}{n-1}}{\frac{C}{n} + \left(1 - \frac{C}{n}\right) \cdot \left(1 - \frac{|\cC_1|}{n-1}\right)}
				\ge \frac{\lambda C}{2n}. \qedhere
			\end{align*}
		\end{proof}
		
		By Chernoff's bound and $\delta(G_0[M_1,R]) \ge \delta \lambda n/4$, with high probability $\delta(H_0) \ge \delta \lambda^2 n/8$. Fix such an outcome of $H_0$, write $H := H_0 \cup \rH_2$, and fix a coupling so that $H \subseteq H_0 \cup \rH_1$.
		Then we know that all edges in $H_0 \cup \rH_2$ are coloured in $\cC_1$, but we have not yet revealed the colours. Thus $H$ is coloured uniformly in $\cC_1$.

		We are done if we can find with high probability a rainbow perfect matching in \(H\) with colours in $\cC_1$.
		To this end, we define the following auxiliary digraph 
		\(D\) with vertex set \([m]\), where $m := |R| = |M_1| = |\cC_1| = (\lambda - \eps)n$.
		Let \(\sigma_1:  [m] \to M_1 \) and \(\sigma_2:  [m] \to R\) be arbitrary bijections.
		Let \(D_0\) and \(\rD_2 \) be the digraphs on \([m]\) with the following edges:
		for distinct $x,y \in [m]$ we have \(xy \in E(D_0) \) if and only if 
		\(\sigma_1 (x) \sigma_2 (y) \in E(H_0)\), and
		\(xy \in E(\rD_2)\) if and only if
		\(\sigma_1 (x) \sigma_2 (y) \in E(\rH_2)\).
        Then define $D:=D_0 \cup \rD_2$.
		It is easy to check that $\delta^0(D_0)=\delta(H_0) \ge \delta \lambda^2 m/8$, each directed edge is present in \(E(\rD_2)\) independently with probability at least \(\lambda^2 C/(4m)\), and each edge of $D$ is coloured independently and uniformly at random in \(\cols_1\).

		Therefore, \(D\) satisfies \Cref{thm:directed_ham_cycle_perturbed_graph} and thus, with high probability, \(D\) has a rainbow directed Hamilton cycle 
		\((x_1,\hdots, x_{m})\).
		Then, from the definition of \(D\), it follows that
		\[
		\{\sigma_1(x_1) \sigma_2(x_2), \sigma_1(x_2) \sigma_2(x_3), \hdots, \sigma_1(x_{m})\sigma_2(x_1)\}
		\] 
		is a rainbow perfect matching in \(H\) that uses all colours in $\cC_1$, as desired.   
	\end{proof}

\section{Embedding trees with many bare paths}
\label{sec:many_paths}
\begin{proof}[Proof of \Cref{thm:many_paths}]
    Let $\mu, \nu$ satisfy
    \[
        1/C \ll \zeta \ll \mu \ll \nu \ll \delta, 1/d\, ,
    \]
    such that
    $\mu/\zeta, \mu n, \zeta n/24$ are integers.
    Set $V:=V(G)$, \(r:= \zeta n /24 \in \nats\) and 
    \(k:= 24 (2\mu - \zeta)/\zeta \in \nats.\) Pick $C'$ so that $(1 - C'/n)^2 = 1 - C/n$ (then $C' \ge C/2$).

    Let \(\rg_1, \rg_2 \sim \gnp{n}{C'/n}\) and \(\rg_3 \sim \gnp{n}{1/2}\) be independent binomial random graphs on \(V\), so that \(G\sim G_0 \cup \rg_1 \cup \rg_2\).
    Let \(G_0'\), \(\rg_1'\)
    be the subgraphs of \(G_0, \rg_1\) respectively 
    with edges in \(E(\rg_3)\), and let \(G' \sim G_0' \cup \rg_1'\).
    Let \(G\dprm\) be the spanning graph of \(K_n\) with edges disjoint from
    \( E(\rg_3) \cup E(\rg_1)\).
    Straightforward applications of the union bound and Chernoff's bound yield that, with high probability, simultaneously
    \setlist{nolistsep}
    \begin{enumerate}[label = \rm(\roman*)]
        \item 
			\(G_0'\) has minimum degree at least \(\delta/3\),
		\item \label{itm:pseudorandom}
			every disjoint vertex sets $A, B$ in \(G\dprm\) with $|A| |B| > 240n$ satisfy $e_{G''}(A,B) \ge |A| \, |B| / 3$.
    \end{enumerate}

	In particular, every induced subgraph of $G''$ on at least $0.99n$ vertices is pseudorandom (as $250 \cdot 0.99 n \ge 240 n$).
    Notice that \(G'\) and \(G\dprm\) are edge-disjoint since \(E(G') \subseteq E(\rg_3)\) and \(E(G\dprm) \cap E(\rg_3) = \emptyset\).

    First we will use \(G'\) to construct the absorbing structure.
    Then we will use the random subgraph of \(G\dprm\), with edges in \(E(\rg_2)\) and colours disjoint from those on the absorbing structure,
    to embed the almost spanning forest resulting from removing \(r=\zeta n/24\) bare paths (whose precise length will be determined later). 
	Finally, we will use the absorbing structure to complete this to a rainbow copy of $T$.

    We start by setting aside the sets needed to build our absorbing structure.
    By the union bound, the conclusions of \Cref{lem:flexible_sets,lem:connecting_gadgets,lem:gadget_exists} hold simultaneously and with high probability for \(G'\), so we assume they all hold.
    Let $V_{\flex} \subseteq V$ and $\cC_{\flex} \subseteq \cC$ be the sets of size $\mu n$ given by \Cref{lem:flexible_sets}.
    Let $V_{\buf}$ and $\cC_{\buf}$ be arbitrary subsets of $V \setminus V_{\flex}$ and $\cC \setminus \cC_{\flex}$ respectively, each of size $\mu n$.
    Let $X:=\{x_i:i \in [r]\}$, $Y:=\{y_i:i \in [r]\}$ and $W$ be pairwise disjoint subsets of $V \setminus (V_{\flex} \cup V_{\buf})$ with $|W|=(2\mu-\zeta)n$.
    Let $\cD$ be a subset of $\cC \setminus (\cC_{\flex} \cup \cC_{\buf})$ with $|\cD|=(2\mu-\zeta)n$.
    Our absorber will be able to absorb a small set of vertices and a small set of colours while connecting, for each $i \in [r]$, the vertex $x_i$ to the vertex $y_i$ through a rainbow path.
    The vertex set $W$ and the colour set $\cD$ will only be used to make our approach work and do not play a special role.

    Let $H$ be the $(\mu n, \zeta/\mu)$-template graph on $(R,S_1 \cup S_2)$ given by \Cref{lem:template}, with $|R|=(2\mu-\zeta) n$ and $|S_1|=|S_2|=\mu n$, and notice $e(H)=40 \cdot |R|$.
    Let $\pi_{v}:S_1 \cup S_2 \rightarrow V_{\buf} \cup V_{\flex}$ be a bijection such that $\pi_v(S_1)=V_{\buf}$ (and thus $\pi_v(S_2)=V_{\flex}$).
    Similarly, let $\pi_c:S_1 \cup S_2 \rightarrow \cC_{\buf} \cup \cC_{\flex}$ be a bijection such that $\pi_c(S_1)=\cC_{\buf}$ (and thus $\pi_c(S_2)=\cC_{\flex}$).    
    Observe that we have $|W|=|\cD|=|R|=r \cdot k$ (recalling that $r = \zeta n / 24$ and $k = 24(2\mu-\zeta)/\zeta$), and thus we can write \(W=\{w_x :x \in R\}\) and \(\cD=\{d_x: x \in R\} \).
    
    \begin{claim}
    \label{claim:gadgets}
        With high probability,
        \(G'\) contains a $(\pi_v(y),d_x)$-gadget and a $(w_x,\pi_c(y))$-gadget, for each $xy \in E(H)$ with $x \in R$ and $y \in S_1 \cup S_2$, with the following property.
        The internal vertices of any two of them are pairwise disjoint and disjoint of $V_{\buf} \cup V_{\flex} \cup X \cup Y \cup W$; similarly, the internal colours of any two of them are pairwise disjoint and disjoint of $\cC_{\buf} \cup \cC_{\flex} \cup \cD$.
    \end{claim}

    \begin{proof}
        Let $\mathcal{A}$ be a maximal collection of gadgets as in the statement of the claim and suppose for contradiction $|\mathcal{A}| < 2 |E(H)| = 2 \cdot 40 \cdot |R| = 80 (2\mu - \zeta)n$.
        Let $V_0$ (resp.\ $\cC_0$) be the union of the vertices (resp.\ colours) spanned by the gadgets in $\mathcal{A}$ and those in $V_{\buf} \cup V_{\flex} \cup X \cup Y \cup W$ (resp.\ $\cC_{\buf} \cup \cC_{\flex} \cup \cD$).
        Then $|V_0|,|C_0| = O(\mu n) < \nu n$, where we used $\mu \ll \nu$.
        Hence, by the conclusion of \Cref{lem:gadget_exists} for \(G'\), we can add another gadget to $\mathcal{A}$, contradicting its maximality.
    \end{proof}

    Partition $R$ into $r$ sets $R_1,\dots,R_r$ each of size $\abs{R}/r = k$
    and let $A_i^1,\dots,A_i^{80k}$ 
    be the gadgets for the edges incident to \(R_i\).
    Note that there are precisely $80k$ of them since $d_H(x)=40$ for each $x \in R$, and each edge incident to \(x\) has two associated gadgets.
    We will now connect \(x_i\) to \(y_i\) via short rainbow paths and the gadgets associated to \(R_i\).
    \begin{claim}
    \label{claim:connecting}
        With high probability, \(G'\) contains a collection $\{P_i^1, \dots, P_i^{80k+1}: i \in [r]\}$ of $(80k+1) \cdot r$ rainbow paths of length three such that the following holds.
        All their interior vertices (resp.\ colours) are distinct, and disjoint from $V_{\buf} \cup V_{\flex} \cup X \cup Y \cup W$ (resp.\ $\cC_{\buf} \cup \cC_{\flex} \cup \cD$) and the set of vertices (resp.\ colours) spanned by the gadgets given by \Cref{claim:gadgets}.
        Moreover, for each $i \in [r]$, the path $P_i^1$ starts with $x_i$ and ends with the first vertex of $A_i^1$; the path $P_i^j$ starts with the last vertex of $A_i^{j-1}$ and ends with the first vertex of $A_i^j$, for each $2 \le j \le 80k$; the path $P_i^{80k+1}$ starts with the last vertex of $A_i^{80k}$ and ends with $y_i$.
    \end{claim}
    \begin{proof}
        Let $\cP$ be a maximal collection of connecting paths as in the statement of the claim.
        Suppose for contradiction
        $|\cP| < (80k+1) \cdot r$ and let
        $u$ and $v$ be a pair of vertices as in the statement of the claim not connected by any path in \(\cP\).
        Let $V_0$ (resp.\ $\cC_0$) be the union of the vertices (resp.\ colours) spanned by the connecting paths of $\cP$, the gadgets given by \Cref{claim:gadgets} and the vertices in $V_{\buf} \cup V_{\flex} \cup X \cup Y \cup W$ (resp.\ the colours in $\cC_{\buf} \cup \cC_{\flex} \cup \cD$).
        Observe that $|V_0|,|\cC_0| = O(\mu n) < \nu n$, using \(\mu \ll \nu\).
        By the conclusion of \Cref{lem:connecting_gadgets} for \(G'\), there exists a rainbow path of length three with endpoints $u$ and $v$, that avoids $V_0$ and $\cC_0$.
        Therefore we can add another connecting path to $\cP$, contradicting its maximality.
    \end{proof}
    Recall that for each pair $(x_i,y_i)$ we built $80k$ gadgets and $80k+1$ connecting paths.
    As we will see shortly, when using the absorbing structure, for each pair we will use the absorbing paths (each of length $10$, c.f.~\Cref{fig:absorber}) in precisely \(2k\) gadgets, and the avoiding path (each of length $9$, c.f.~\Cref{fig:absorber}) in every other gadget.
    Together with the connecting paths (each of length $3$), we will get a path of length
    \begin{equation} \label{eq:length-xi-yi}
        \ell := 3(80k+1) + 9 (80k-2k) + 10\cdot 2k = 962 k + 3
    \end{equation}
    between \(x_i\) and \(y_i\).
    Let 
    \(\ell' := 28\) and let
    \(F\) be the forest resulting from removing the internal vertices of \(r\) bare paths of length \(\ell + \ell'\) from \(T\).
    Observe that \(T\) does indeed have \(r\) bare paths of length \(\ell+\ell'\), since
    $\ell + \ell'
    = 962 k + 31
    \le 1000 k 
    = 1000 \cdot 24(2\mu-\zeta) \zeta^{-1}
    \le 6\zeta^{-1}$.
    We removed the internal vertices of paths of length $\ell+\ell'$ as opposed to length $\ell$, as this will allow us to cover, using \Cref{lem:flexible_sets}, the leftover vertices and colours after the embedding of \(F\) in \(G\dprm\) via \Cref{thm:rainbow-almost-spanning-embedding}.
	In fact, $\ell'$ has been chosen so that the path length we get from \Cref{lem:flexible_sets} matches the length of the paths removed from \(T\), and \emph{only this choice of \(\ell'\) works}.

    Let $\tilde{V}$ (resp.\ $\tilde{\cC}$) be the set of all vertices (resp.\ colours) used to build the gadgets of \Cref{claim:gadgets} and the connecting paths of \Cref{claim:connecting}.
    Then
    \begin{align*}
            |\tilde{V}| 
            &= 2 e(H) \cdot 10 + \abs{\Vflx} + \abs{\Vbuf} + \abs{W} + \abs{X} + \abs{Y} + (80k+1)\cdot r \cdot 2\\
            &= (962k + 28)r 
    \end{align*}
    and similarly
    \[
        |\tilde{\cols}|
           = 2 e(H) \cdot 9 + \abs{\colsflx} + \abs{\colsbuf} + \abs{D} + (80k+1)\cdot r \cdot 3
           = (962 k + 27)r\, .
    \]
    Moreover, since we removed $(\ell + \ell' -1) \cdot r$ vertices from $T$, we have
    \[
		\abs{V(F)} 
	  = n - r(\ell+\ell' - 1) 
	  = n- (962k+30)r.
	\]

    \begin{claim}
    \label{claim:F-embedding}
        With high probability, \(G\dprm[V\setminus \tilde{V}]\) contains a rainbow copy of \(F\), with  edges in \(E(\rg_2)\) and colours disjoint from \(\tilde{\cols}\).
    \end{claim}
    \begin{proof}
		Let $\cols_0$ be a subset of $\cols \setminus \tilde{\cols}$ of size $|V \setminus \tilde{V}|$ (such a set exists as $|\cC \setminus \tilde{\cols}| \ge |V \setminus \tilde{V}|$), and write $V_0:= V \setminus \tilde{V}$.
		Let $H := G''[V_0]$. By \ref{itm:pseudorandom}, $H$ is pseudorandom. Let $H'$ be the subgraph of $H$, consisting of edges that are in $\rg_2$ and are coloured in $\cols_0$. This means that each edge of $H$ is included in $H'$ with probability $p = (C'/n) \cdot |\cols_0|/|\cols|$.
        Now reveal the colouring of \(H'\), and observe that it is distributed uniformly at random in $\cols_0$.
        Moreover we have $|V(F)| = |V_0| - 2r = |V_0| - \zeta n / 12 \le (1 - \zeta/12)|V_0|$ and $|\cols_0| = |V_0|$.
        Hence, we can apply
        \Cref{thm:rainbow-almost-spanning-embedding} to $H'$ and get that, with high probability,
        \(H'\) contains a rainbow copy of 
        \(F\).
    \end{proof}

	Fix an embedding of $F$ in $G''$.
    Let 
    \(V'\) and
    \(\cols'\)
    be the vertices and colours not in $\tilde{V}$ and $\tilde{\cC}$ and not used for the embedding of $F$.
    Then $\abs{V'} = n-|\tilde{V}|-\abs{V(F)}=2r$ and $\abs{\cols'} = (n-1)-|\tilde{\cC}|-\left(\abs{V(F)}-r-1\right) = 4r$, where we used that $F$ is a forest
    with $r+1$ components.
    Thus we can write $V'=\{v_i^1, v_i^2: i \in [r]\}$ and
    \(\cols' = \{ c_i^{11}, c_i^{12}, c_i^{21}, c_i^{22}: i \in [r]\}\).
    Let $x_i'$ and $y_i'$ be the embedded endpoints of the $i$-th bare path of $T$.
    
    \begin{claim} \label{claim:cover-leftover}
        With high probability, there is a collection of pairwise vertex- and colour-disjoint rainbow paths \(\{ Q_i^{11}, Q_i^{12},Q_i^{21},Q_i^{22}: i\in [r]\}\) such that
        \begin{itemize}
            \item the ends of $Q_i^{11}$ are $x_i',v_i^1$, of $Q_i^{12}$ are $v_i^1, x_i$, of $Q_i^{21}$ are $y_i',v_i^2$, and of $Q_i^{22}$ are $v_i^2,y_i$;
            \item each path has length \(7\) and all its $6$ internal vertices in $V_{\flex}$;
            \item for each $s,t \in \{1,2\}$, the path \(Q_i^{st}\) has colours in $\cC_{\flex} \cup \{c_i^{st}\}$ and uses the colour $c_i^{st}$.
        \end{itemize}
    \end{claim}
    \begin{proof}
        The collection can be found greedily using \Cref{lem:flexible_sets}.
        Indeed suppose we are not done yet and we still need to make a connection, say from $x_i'$ to $v_i^1$.
        Let \(V_0 \subseteq \Vflx, \cols_0\subseteq \colsflx\) be the vertices and colours already used.
        Then, using \(r = \zeta n /24\), we have
        \(\abs{V_0} = \abs{\cols_0} <  24r=\zeta n\).
        Therefore, we can apply the conclusion of \Cref{lem:flexible_sets} and find a rainbow path $Q_i^{ii}$ of length 7 with endpoints $x_i'$ and $v_i^1$, internal vertices in $V_{\flex} \setminus V_0$ and colours in $(\cC_{\flex} \setminus C_0) \cup \{c_i^{11}\}$, that uses the colour $c_i^{11}$.
        Therefore the claim holds.
    \end{proof}

    Let \(\Vflx', \colsflx'\) be the vertices in \(\Vflx\) and colours in \(\colsflx\) used in the paths in \Cref{claim:cover-leftover}.
    Then \(\abs{\Vflx'} = \abs{\colsflx'} = 24r = \zeta n\).

    Let \(H_1 = H - \pi_v^{-1}\left( \Vflx'\right)\) 
    and \(H_2 = H - \pi_v^{-1} \left( \colsflx' \right)\).
    Then by choice of $H$ according to \Cref{lem:template}, \(H_1\) and \(H_2\) have perfect matchings \(M_1\) and \(M_2\) respectively.
    Recall that each pair \((x_i, y_i)\) is associated with the gadgets \(A_i^1, \hdots, A_i^{80k}\), two for every edge incident to each \(x\in R_i\).
    For \(xy \in E(H_1)\) with $x \in R$ and $y \in S_1 \cup S_2$, let \(P_{M_1}(xy)\) be the absorbing path of the \((\pi_v(y), d_x)\)-gadget if \(xy \in E(M_1)\), and the avoiding path otherwise.
    Similarly, for \(xy \in E(H_2)\) with $x \in R$ and $y \in S_1 \cup S_2$, let \(P_{M_2}(xy)\) be the absorbing path of the \((w_x, \pi_c(y)\)-gadget if \(xy \in E(M_2)\), and the avoiding path otherwise.
    Then define
    \[
    P_i:=   \bigcup_{j\in [80k+1]}\, P_i^{j}\quad \cup 
          \bigcup_{xy \in E(H_1):\, x\in R_i} P_{M_1}(xy) \quad \cup
          \bigcup_{xy \in E(H_2):\, x\in R_i} P_{M_2}(xy)\, ,
    \]
    and observe that $P_i$ is a path with endpoints $x_i$ and $y_i$. 
    Each \(x\in R_i\) lies in exactly one edge of \(M_1\) and exactly one edge of \(M_2\), so $P_i$ consists of \(2 |R_i| = 2k\) absorbing paths (of length $10$), \(80k -2k\) avoiding paths (of length $9$), and $80k+1$ connecting paths (of length $3$).
    Hence, \(P_i\) has length \(\ell\) (c.f.\ \Cref{eq:length-xi-yi}).
    Finally define
    \[
    P_i' := (Q_i^{11} \cup Q_i^{12}) \cup P_i \cup (Q_i^{21} \cup Q_i^{22})\, , 
    \]
    and observe that $P_i'$ is path with endpoints $x_i'$ and $y_i'$ and has length \(\ell + 4\cdot 7 = \ell +\ell'\).
    By construction, the paths \(P_i'\) are rainbow, pairwise vertex- and colour-disjoint, and use no vertex or colour in the embedding of \(F\).
    
    Recall that we obtained $F$ from $T$ by removing the internal vertices of $r$ bare paths of length $\ell+\ell'$ and that $x_i'$ and $y_i'$ were the images of the endpoints of the $i$-th bare path.
    Therefore the image of $F$ together with \(\bigcup_{i=1}^{r} P_i'\) gives a rainbow embedding of \(T\) in $G$, as desired.
\end{proof}

\section{Concluding remarks}
\label{sec:concluding}
    In this paper we studied the problem of finding rainbow subgraphs of uniformly coloured randomly perturbed graphs.
    First, we gave a general result (\Cref{thm:mc_diarmid_argument}) applicable when the number of colours is asymptotically optimal.
    Then, we gave a result concerning rainbow bounded-degree spanning trees when the number of colours is exactly optimal (\Cref{thm:main}).
    We showed that any given bounded-degree spanning tree typically appears when a linear number of random edges are added to any dense graph, and all the edges are uniformly coloured with colours in $[n-1]$.
    It would be interesting to improve our result to a \emph{universality} statement: namely, is it true that a uniformly coloured randomly perturbed graph with colours in $[n-1]$ typically contains a rainbow copy of every bounded-degree spanning tree at once?
    The uncoloured universality question was already considered by B\"{o}ttcher, Han, Kohayakawa, Montgomery, Parczyk and Person~\cite{tree_universality}, who proved that for every $\alpha \in (0,1)$ and $d \in \mathbb{N}$ there exists $C=C(\alpha,d)>0$ such that, if $G_0$ is an $n$-vertex graph with $\delta(G_0) \ge \alpha n$, then with high probability $G_0 \cup \rg(n,C/n)$ contains every $n$-vertex tree $T$ with $\Delta(T) \le d$.
    We remark that the edge density is optimal (up to a constant factor).

\bibliographystyle{amsplain} 
\bibliography{bibliography}

\appendix 

\section{Appendix - Proof of \Cref{prop:sample-pseudorandom-tree}} \label{appendix}

	The proof of \Cref{prop:sample-pseudorandom-tree} uses the following result of Alon, Krivelevich and Sudakov~\cite[Theorem 1.4]{AKS}, for which we need the following definition.
	Given two positive numbers $c$ and $\alpha<1$, a graph $G$ is called an \emph{$(\alpha,c)$-expander} if every subset of vertices $X \subseteq V(G)$ with $|X|\le \alpha |V(G)|$ satisfies $|N_G(X)| \ge c |X|$.

	\begin{theorem}[Theorem 1.4 in~\cite{AKS}]
	\label{thm:AKS_expander_contains_trees}
		Let $d \ge 2$ be an integer and $0 < \eps <1/2$.
		Then for $n$ large enough the following holds.
		Let $G$ be a graph on $n$ vertices of minimum degree $\delta$ and maximum degree $\Delta$ such that, with $K := \frac{20d^2 \log(2/\eps)}{\eps}$, we have
		\begin{enumerate}
			\item \label{eq:AKS_condition1} $\Delta^2 \le \frac{1}{K}\exp\left(\frac{\delta}{8K}-1\right)$, and
			\item \label{eq:AKS_condition2} every induced subgraph $G_0$ of $G$ with minimum degree at least $\frac{\delta}{2K}$ is a $\left(\frac{1}{2d+2},d+1\right)$-expander.
		\end{enumerate}
		Then $G$ contains a copy of every tree $T$ on at most $(1-\eps)n$ vertices of maximum degree $d$.
	\end{theorem}

	Using \Cref{thm:AKS_expander_contains_trees}, the proof of \Cref{prop:sample-pseudorandom-tree} reduces to verifying that the conditions of \Cref{thm:AKS_expander_contains_trees} hold with high probability for a $C/n$-random subgraph of a pseudorandom graph $G$.
	The argument and calculations follow very closely those for the proof of Theorem 1.1 in~\cite{AKS}, where this was verified for \(\gnp{n}{C/n}\), which can be seen as the \(C/n\)-random subgraph of the complete graph on $n$ vertices.

	We first show that a random subgraph of a pseudorandom graph with high probability contains a nearly spanning subgraph with good local expansion property.

	\begin{lemma}[{c.f.\ \cite[Lemma 3.1]{AKS}}]\label{lemma:similar-to-AKS}
		Let $1/C \ll \eps, 1/d$ and set \(\theta := 0.01\eps\) and \(D := C/10\). 
		Let $G$ be a pseudorandom graph on $n$ vertices and $p:=C/n$.
		Then, with high probability, \(G_p\) contains a subgraph \(G^\ast\) that satisfies the following.
		\begin{enumerate}
			\item \(\abs{V(G^\ast)}\ge (1-\theta)n\);
			\item \(D \le \deg_{G^\ast}(v) \le 25 D\), for all \(v \in V(G^\ast)\);
			\item Every induced subgraph  \(G_0\) of \(G^\ast\) with \(\delta(G_0) \ge 100d \log D\) is a \((\frac 1 {2d+2}, d+1)\)-expander.
		\end{enumerate}
	\end{lemma}

	Then we state easy facts about random subgraphs of pseudorandom graphs.

		\begin{lemma}[{c.f.\ \cite[Proposition 3.2]{AKS}}] \label{lemma:simple-edge-bounds}
		Let $G$ be a pseudorandom graph on $n$ vertices and $p =p(n) \in (0,1)$ such that $np>20$. With high probability the following holds for $G_p$.
		\begin{enumerate}
			\item \label{itm:simple-a}
				For any disjoint vertex subsets \(A\) and \(B\) with \(\abs{A}=a\), \(\abs{B}=b\) and \(a b p \ge 250 n\), the number of edges between them is at least \( abp/4  \) and at most \(3 abp/2\).
			\item \label{itm:simple-b}
				Every subset of vertices \(A\) of size \(a \le n/4\) spans at most \(apn/2\) edges.
		\end{enumerate}
		\end{lemma}
		\begin{proof}
			For the first part of the claim note that, because \(G\) is pseudorandom and $ab \ge 250 n$, we have
			\(e_G(A,B) \ge a b /3\), so 
			\(\EE[e_{G_p}(A,B)] \ge p a b /3\), and then the lower bound follows from a standard application of Chernoff's bound and the union bound over the at most \(2^{2n}\) choices for \(A,B\).
			For the upper bound,
			the trivial upper bound \(e_{G}(A,B) \le a b\)
			implies \(\EE[e_{G_p}(A,B)] \le a b p\)
			and then again a standard application of Chernoff's bound and the union bound yield the desired result.
		
			The second part follows directly from the second part of \cite[Proposition 3.2]{AKS} because we can couple $G$ with $\gnp{n}{p}$ so that $G \subseteq \gnp{n}{p}$.
		\end{proof}

	We now deduce \Cref{prop:sample-pseudorandom-tree} from \Cref{lemma:similar-to-AKS}.

	\begin{proof}[Proof of \Cref{prop:sample-pseudorandom-tree}]
		Let $1/C \ll \eps,1/d$ and set \(\theta := 0.01\eps\) and \(D := C/10\).
		Let $G^\ast$ be the subgraph of $G_p$ given by \Cref{lemma:similar-to-AKS}.
		We verify that \(G^\ast\) satisfies the conditions of \Cref{thm:AKS_expander_contains_trees} with parameters \(d\) and 
		\(\eps_1 := \frac{\eps - \theta}{1-\theta} \in [0.99\eps, \eps]\), and so
		\(K = K(\eps,d) = 20 \eps_1^{-1} d^2\log \left( 2 /\eps_1\right)\).
		
		Using that
		\(
		D \le \delta(G^\ast) \le \Delta(G^\ast) \le 25 D
		\), $D=C/10$ and $1/C \ll \eps, 1/d$, we have 
		\[
			(\Delta(G^\ast))^2 \le 625 D^2 \le \frac{1}{K}\exp\left(\frac{D}{8K}-1\right) \le \frac{1}{K}\exp\left(\frac{\delta(G^\ast)}{8K}-1\right)
		\]
		and hence the first condition is satisfied.
		
		To check the second condition, it suffices to show that 
		\[
		\frac{\eps_1\, D}{40 d^2 \log \left(2/\eps_1\right)}
		\ge 100 d \log D,
		\]
		which follows from the fact that  \(1/C \ll \eps_1, 1/d\) and that the function \( D/\log D\) is increasing.
		Hence \(G^\ast\) contains a copy of every tree of maximum degree \(d\) and of up to
		\((1-\eps_1)(1-\theta)n \ge (1-\eps)n\) vertices.
	\end{proof}

	Finally we prove \Cref{lemma:similar-to-AKS}
	\begin{proof}[Proof of \Cref{lemma:similar-to-AKS}]
		Assume that the conclusions of \Cref{lemma:simple-edge-bounds} hold.
		Let \(X\) be the set of \(\theta n /2\) vertices of largest degree in \(G_p\).
		Part \ref{itm:simple-b} of \Cref{lemma:simple-edge-bounds} implies that the number of edges in \(G_p[X]\) is at most \(\abs{X} p n/2 = 5D\abs{X}\).
		On the other hand, since
		\(
		p \abs{X} (n-\abs{X}) \ge 2 D \theta n \ge 250 n,
		\)
		using \(D^{-1} = 10 C^{-1} \ll \eps = 100 \theta\), Part \ref{itm:simple-a} of \Cref{lemma:simple-edge-bounds} implies that, with high probability, 
		\(e_{G_p}(X, V(G)\setminus X ) \le 3 \abs{X} n p/2 = 15 D\abs{X}\).
		Hence, 
		\(\sum_{v\in X} \deg_{G_p} (v) \le 25 D \abs{X}\), which implies there is a vertex in \(X\) with degree at most \(25 D\).
		By the definition of \(X\), it follows that \(G_p\) has at most \(\theta n /2\) vertices of degree greater than \(25 D\).

		Delete these vertices, denote the remaining graph by \(G'\) and observe $|V(G')| \ge (1-\theta/2)n$. We greedily remove from \(G'\) vertices of degree less than \(D\), until none are left.
		Suppose that we deleted at least \(\theta n /2\) vertices.
		Let \(Y\) be a subset of size \(\theta n /2\) of the deleted vertices.
		Then \(\deg_{G_p}(y, V(G')\setminus Y) \le D\) for each $y \in Y$,
		so 
		\(e_{G_p}(Y, V(G')\setminus Y) \le D \abs{Y}\).
		On the other hand, 
		\(
		p \abs{Y} \abs{V(G') \setminus Y} 
		\ge 5 D \theta (1-\theta) n
		\ge 250 n
		\)
		(using \(D^{-1} \ll \theta\)), so Part \ref{itm:simple-a} of \Cref{lemma:simple-edge-bounds} implies that 
		\(
		e_{G_p}(Y, V(G') \setminus Y) \ge \abs{Y} \abs{V(G') \setminus Y} p/4 > 2 D \abs{Y}\), which is a contradiction.
		Hence, the number of vertices that we deleted is at most \(\theta n /2\).
		Denote the resulting graph by \(G^\ast\) and observe that it satisfies the first two properties of the lemma.

		Suppose \(G^\ast\) fails to satisfy the third property of the lemma.
		Then there exist \(U\subseteq V(G^\ast)\) such that 
		\(G^\ast[U]\) has minimum degree at least \(100d\log D\) and is not a \((\frac 1 {2d+2}, d+1)\)-expander.
		This implies there is \(X\subseteq U\), of size 
		\(t\le \frac 1 {2d+2} \abs{U}\), such that for the set \(Y:=N_{G^\ast[U]}(X)\) it holds that \(\abs{Y} \le (d+1) t\).
		
		We first consider the case \(t \le \frac{\log D}{D} n\).
		By the minimum degree condition, we have
		\(e_{G_p}(X,Y) \ge 50d t \log D\).
		Let \(A_t\) be the event that there exist vertex sets 
		\(X, Y\) with \(\abs{X} = t\), \(\abs{Y} \le (d+1) t\), and
		\(e_{G_p}(X,Y) \ge 50d t \log D\).    
		Then we can bound \(\PP[A_t]\) as follows, where we remark that for the first inequality we use that the binomial coefficient is increasing until the middle layer and that for the penultimate inequality we use \(t < (d+1) t < (d+1) \frac{\log D}{D}n < n/2\) since \(D^{-1} = 10 C^{-1} \ll d\), and $\log(3e/10) \le -0.1$. 
		\begin{align*}
			\PP[A_t] &\le \binom{n}{t} \binom{n}{(d+1) t} \binom{t (d+1) t}{50 d t \log D} p^{50 d t \log D} \\
			&\le \left( 
			\frac{\e n}{t} \cdot 
			\left( \frac{\e n}{(d+1)t} \right)^{d+1} \cdot 
			\left( 
			\frac{\e (d + 1) t^2 p}{50dt \log D}
			\right)^{50 d \log D}
			\right)^t  \\
			&\le \left(
			\e \cdot \left(\frac{n}{t}\right)^{2d} \cdot 
			\left( \frac \e {d+1} \right)^{d+1}
			\left( \frac{t/n}{\log D /D} \right)^{50d \log D}
			\cdot \left( \frac{\e (d+1)}{5d} \right)^{50d \log D}
			\right)^t
			\\
			&\le
			\left( 
			\e \cdot \left( \frac{t/n}{\log D /D} \right)^{50d \log D  -2d}
			\cdot \left(\frac{D}{\log D}\right)^{2d}
			\cdot \left(\frac{3 \e}{10}\right)^{50 d \log D}
			\right)^t
			\\
			&\le 
			\left(
			\e^{-5d \log D + 2d\log D +1 }
			\cdot 
			\left( \frac{t/n}{\log D /D} \right)^{40d \log D}
			\right)^t
			\\
			&\le 
			\left( 
			\e^{-2d \log D} \cdot \left( \frac{t/n}{\log D /D} \right)^{40d \log D}
			\right)^t.
		\end{align*}
	For \(t< \log n\), this is at most
	\[
	\left( \frac{\log n/n}{\log D /D} \right)^{40d \log D}
	\le n^{-79}
	\]
	and for \(\log n \le t \le \frac {\log D} D n\) it is at most
	\[
	\e^{-2d \log D \cdot \log n} \le n^{-4}
	\]
	so in either case \(\PP[A_t] = o(n^{-1})\).

	Finally we consider the case \(t\ge \frac {\log D} D \cdot n\).
	Set \(Z := U \setminus (X \cup Y)\), and notice that \(e_{G_p}(X, Z) = 0\), since \(G^\ast[U]\) is an induced subgraph of \(G_p\).
	From \(\abs{U} \ge (2d + 2)t \), $|X|=t$ and \(\abs{Y} \le (d+1) t\) it follows that
	\(\abs{Z} \ge dt\).
	Let \(B_t\) be the event that there exist vertex subsets of size \(t\) and \(dt\), with no edge between them in \(G_p\).
	Then
	\begin{align*}
		\PP[B_t] &\le \binom{n}{t} \binom{n}{dt} (1-p)^{dt^2}
		\\
		&\le \left( \frac{\e n}{t} \cdot \left(  \frac{\e n}{dt}\right)^d \cdot \e^{-pdt} \right)^t\\
		&\le \left( 
		\left( \frac {\e n} t\right)^{2d}  \cdot \e ^{-pdt} 
		\right)^t
		= \left( 
		\left( \frac {\e n} t \right)^{2} \e ^{-pt}  \right)^{dt}
		\\
		&\le \left( 
		\left( \frac {\e n} {n \log D /D} \right)^{2}
		\e^{-\frac{10 D}{n} \cdot \frac{\log D}{D} n}
		\right)^{dt}\\
		&\le \left( D^2 D^{-10} \right)^{dt} = o(n^{-1}).
	\end{align*}
	Hence \(G^\ast\) fails to satisfy the third condition of the lemma with probability at most $\sum_{t=1}^{n} (\PP[A_t] + \PP[B_t])
	=o(1)$.

	Therefore, by the union bound, we conclude that with high probability \(G^\ast\) satisfies all conditions of the Lemma.
	\end{proof}

\end{document}